\documentclass[a4paper,10pt]{article}
\usepackage{bbm}
\usepackage{amsfonts}
\usepackage{amssymb}
\usepackage[all]{xy}
\usepackage{amsthm}
\usepackage{mathrsfs}
\usepackage{latexsym}
\usepackage{amsthm,a4wide,color}
\usepackage{amsmath,amscd,verbatim}
\usepackage{hyperref}
\input diagxy
\usepackage{tikz,CJK}

\theoremstyle{plain}
  \newtheorem{thm}{Theorem}[section]
  \newtheorem{lem}[thm]{Lemma}
  \newtheorem{prop}[thm]{Proposition}
  
\theoremstyle{definition}
  \newtheorem{defn}[thm]{Definition}
  
  \newtheorem{exmp}[thm]{Example}

\begin{document}
\newcommand\thda{\mathrel{\rotatebox[origin=c]{-90}{$\twoheadrightarrow$}}}
\newcommand\thua{\mathrel{\rotatebox[origin=c]{90}{$\twoheadrightarrow$}}}

\newcommand{\oto}{{\lra\hspace*{-3.1ex}{\circ}\hspace*{1.9ex}}}
\newcommand{\lam}{\lambda}
\newcommand{\da}{\downarrow}
\newcommand{\Da}{\Downarrow\!}
\newcommand{\D}{\Delta}
\newcommand{\ua}{\uparrow}
\newcommand{\ra}{\rightarrow}
\newcommand{\la}{\leftarrow}
\newcommand{\Lra}{\Longrightarrow}
\newcommand{\Lla}{\Longleftarrow}
\newcommand{\rat}{\!\rightarrowtail\!}
\newcommand{\up}{\upsilon}
\newcommand{\Up}{\Upsilon}
\newcommand{\ep}{\epsilon}
\newcommand{\ga}{\gamma}
\newcommand{\Ga}{\Gamma}
\newcommand{\Lam}{\Lambda}
\newcommand{\CF}{{\cal F}}
\newcommand{\CG}{{\cal G}}
\newcommand{\CH}{{\cal H}}
\newcommand{\CI}{{\cal I}}
\newcommand{\CB}{{\cal B}}
\newcommand{\CT}{{\cal T}}
\newcommand{\CS}{{\cal S}}
\newcommand{\CV}{{\cal V}}
\newcommand{\CP}{{\cal P}}
\newcommand{\CQ}{{\cal Q}}
\newcommand{\mq}{\mathcal{Q}}
\newcommand{\cu}{{\underline{\cup}}}
\newcommand{\ca}{{\underline{\cap}}}
\newcommand{\nb}{{\rm int}}
\newcommand{\Si}{\Sigma}
\newcommand{\si}{\sigma}
\newcommand{\Om}{\Omega}
\newcommand{\bm}{\bibitem}
\newcommand{\bv}{\bigvee}
\newcommand{\bw}{\bigwedge}
\newcommand{\lra}{\longrightarrow}
\newcommand{\tl}{\triangleleft}
\newcommand{\tr}{\triangleright}
\newcommand{\dda}{\downdownarrows}
\newcommand{\dia}{\diamondsuit}
\newcommand{\y}{{\bf y}}
\newcommand{\colim}{{\rm colim}}
\newcommand{\id}{\rm id}
\newcommand{\fR}{R^{\!\forall}}
\newcommand{\eR}{R_{\!\exists}}
\newcommand{\dR}{R^{\!\da}}
\newcommand{\uR}{R_{\!\ua}}
\newcommand{\swa}{{\swarrow}}
\newcommand{\sea}{{\searrow}}
\newcommand{\bbA}{{\mathbb{A}}}
\newcommand{\bbB}{{\mathbb{B}}}
\newcommand{\bbC}{{\mathbb{C}}}
\numberwithin{equation}{section}
\renewcommand{\theequation}{\thesection.\arabic{equation}}

\title {Flat ideals in the unit interval with the canonical fuzzy order%\thanks{Supported by the National Natural Science Foundation of China, No. 11771310.}
}
\author{ {Hongliang Lai, Dexue Zhang, Gao Zhang}\\
{\small School of Mathematics, Sichuan University,
Chengdu 610064, China}\\ {\small hllai@scu.edu.cn, dxzhang@scu.edu.cn, gaozhang0810@hotmail.com}}
\date{}
\maketitle

\begin{abstract} A characterization of flat ideals in the unit interval with the canonical fuzzy order is obtained with the help of the ordinal sum decomposition of continuous t-norms. This characterization will be useful in the study of topological and domain theoretic properties  of fuzzy orders.


%03E72 (Fuzzy set theory) 06A99 (order) 03B52 (fuzzy logic)
\end{abstract}
\section{Introduction and preliminaries}
A commutative quantale \cite{Rosenthal1990} is a commutative monoid $(Q,\&, 1)$ such that $Q$ is a complete lattice and
\begin{equation*} p\&\bv_{j\in J}q_j=\bv_{j\in J}p\& q_j, ~ \Big(\bv_{j\in J}q_j\Big)\&p=\bv_{j\in J}q_j\&p. \end{equation*}
for all $p\in Q$ and   $\{q_j\}_{j\in J}\subseteq Q$. The unit $1$ of the monoid $(Q,\&, 1)$ is not necessarily  the top element of $Q$.
Given a commutative quantale $\CQ=(Q,\&, 1)$, since the semigroup operation $\&$ distributes over arbitrary joins, it  determines a binary operation $\ra$ on $Q$ via the adjoint property \begin{equation*} p\&q\leq r\iff q\leq p\ra r. \end{equation*}
The  binary operation $\ra$ is called the \emph{implication}   corresponding to  $\&$. So, commutative quantales are often the table of truth-values in many valued logic, with $\&$ playing the role of the connective ``conjunction", $\ra$ playing the role of the connective ``implication".

A \emph{$\mathcal{Q}$-order} (or an order valued in the quantale $\CQ$) \cite{Zadeh71,Wagner97} on a set $X$ is a reflexive and transitive $\mathcal{Q}$-relation on $X$. Explicitly, a $\mathcal{Q}$-order on  $X$  is a map $P: X\times X\lra Q$ such that $P(x,x)\geq1$   and  $P(y,z)\& P(x,y)\leq P(x,z)$ for any $x,y,z\in X$.  The pair $(X,P)$ is called a  $\mathcal{Q}$-ordered  set.
As usual, we write $X$ for the pair $(X, P)$ and $X(x,y)$ for $P(x,y)$ if no confusion would arise.

%Two elements $x,y$ in a $\CQ$-ordered set $A$ are \emph{isomorphic} if $A(x,y)=A(y,x)=1$. We say that $A$ is \emph{separated} if  isomorphic elements in $A$ are equal, that is, $A(x,y)=A(y,x)=1$ implies that $x=y$.

If $P: X\times X\lra Q$ is a $\CQ$-order on $X$, then $P^{\rm op}: X\times X\lra Q$, given by $P^{\rm op}(x,y)=P(y,x)$, is also a $\CQ$-order on $X$ (by commutativity of $\&$), called the opposite of $P$.

If we let $d_L(p,q)=p\ra q$ for all $p, q\in Q$, then  $d_L$ is a  $\CQ$-order on $Q$, called the canonical $\CQ$-order on $Q$. The opposite $d_R$ of $d_L$ is   given by $d_R(p,q)=q\ra p.$ Both of   $(Q,d_L)$ and $(Q,d_R)$ play important roles in the theory of $\CQ$-ordered sets.

A fuzzy lower set of a $\CQ$-ordered set $ X $ is a map $ \phi: X\lra Q$ such that  $\phi(x)\&X(y,x)\leq\phi(y) $ for all $ x,y\in X $. Dually, a fuzzy upper set of $ X $ is a map $ \psi: X\lra Q $ such that $X(x,y) \&\psi(x)\leq\psi(y)$ for all $ x,y\in X $.

Given a fuzzy lower set $ \phi $ and a fuzzy upper set $ \psi $ of a $\CQ$-ordered set $ X $, the tensor product of $\psi$ and $\phi$ is defined to be the  value \[ \phi \otimes \psi=\bv_{x\in X}\phi(x)\&\psi(x). \] Intuitively, $\phi \otimes \psi$ measures the degree that $ \phi $ intersects with $\psi$.

\begin{defn}(\cite{SV2005,LZZ17})
A  flat ideal in a $\CQ$-ordered $ X $ is a fuzzy lower set  $\phi $ of $X$ such that
\begin{enumerate}\setlength{\itemsep}{-3pt}
		\item[\rm(1)] $ \phi $ is inhabited in the sense that $\bv_{a\in X}\phi(a)\geq1$;
		\item[\rm(2)] $\phi$ is flat in the sense that for any fuzzy upper sets $\psi_1, \psi_2$ of $X$, \[\phi \otimes (\psi_1\wedge\psi_2)= (\phi\otimes\psi_1)\wedge (\phi\otimes\psi_1).  \]
	\end{enumerate}
\end{defn}

Flat ideals are a counterpart of directed lower sets, so, they play a crucial role in the study of the topological and domain theoretic properties of fuzzy orders. It should be noted that there exist different approaches to the notion of ``fuzzy ideals",  a comparative study can be found in \cite{LZZ17}.

In order to understand the structure of flat ideals in fuzzy ordered sets, the first step is to characterize the flat ideals in the table of truth-values, i.e., the flat ideals in   $\CQ$ with the canonical fuzzy order. Among the best known commutative quantales are the unit interval $[0,1]$ together with a continuous triangular norm (t-norm, for short). The importance of such quantales in fuzzy set theory cannot be over estimated. In particular, the  BL-logic of H\'{a}jek \cite{Ha98}, a very successful theory of fuzzy logic, is a logic based on such quantales. The aim of this note is to characterize flat ideals in the unit interval   with the canonical fuzzy order. The result is very likely to be useful in the theories of topological spaces and partially ordered sets based on BL-logic.

\section{Fuzzy lower sets of $([0,1],d_L)$}

A left continuous t-norm \cite{KMP00} on the unit interval $[0,1]$ is a binary operation $\&$ on [0,1] such that $([0,1],\&,1)$ is a commutative quantale. A left continuous t-norm $\&$ is said to be continuous if it is a continuous  map from $[0,1]^2$ to $[0,1]$ with respect to the usual topology.

\begin{exmp}
Basic continuous t-norms and their corresponding implications.
\begin{enumerate}\setlength{\itemsep}{-3pt}
\item[\rm(1)] The t-norm $\min$: \[ a\&b=a\wedge b=\min\{a,b\}, \quad a\ra b=\begin{cases}
		1,&a\leq b,\\
		b,&a>b.
		\end{cases} \]  The implication of the t-norm $\min$ is known as the G\"{o}del implication.
\item[\rm(2)] The product t-norm:  \[ a\&b=ab, \quad a\ra b=\begin{cases}
		1,&a\leq b,\\
		b/a,&a>b.
		\end{cases} \] The implication of the product t-norm is known as the Goguen implication and, as a binary operation, it is  continuous  except at the point $(0,0)$.
\item[\rm(3)] The {\L}ukasiewicz t-norm:\[ a\&b=\max\{0,a+b-1\}, \quad a\ra b=\min\{1-a+b,1\}. \] The implication of the {\L}ukasiewicz t-norm  is known as the {\L}ukasiewicz implication  and it is a continuous binary operation.
	\end{enumerate}
\end{exmp}

The following conclusion is essentially Proposition 2.3 in \cite{KMP00}.

\begin{prop}{\rm(\cite{KMP00})} \label{idempotent}
Let $\&$ be a continuous t-norm on $[0,1]$ and $c$ be an idempotent element of $\&$. Then $x\&y= x\wedge y $ whenever $x\leq c\leq y$.
 \end{prop}

As an immediate corollary of the above proposition we obtain that if $\&$ is a continuous t-norm and $a,b$ ($a<b$) are idempotent elements of $\&$, then the restriction of $\&$ on $[a,b]$ is a continuous t-norm on $[a,b]$, hence $([a,b],\&,b)$ is a commutative quantale.

The following result, known as the ordinal sum decomposition of continuous t-norms, is of fundamental importance in the theory of continuous t-norms.
\begin{thm}\label{ordinal sum decomposition theorem} {\rm(\cite{Mostert1957})}
 For each continuous t-norm $ \& $ on $[0,1]$, there exists a family of pairwise disjoint open intervals $ \{(a_j,b_j)\}_{j\in J} $ of $ [0,1] $ such that
	\begin{enumerate}\setlength{\itemsep}{-3pt}
		\item[\rm(1)] for each $ j\in J $, both $ a_j $ and $ b_j $ are idempotent and the restriction of $ \& $ on $ [a_j,b_j] $ is  either isomorphic to the {\L}ukasiewicz t-norm or to the product t-norm;
		\item[\rm(2)] $x\&y= x\wedge y$ if $(x,y)\notin\bigcup_j[a_j,b_j]^2$.
	\end{enumerate}
\end{thm}

Therefore, for a continuous t-norm $\&$ on $[0,1]$ and  $c\in[0,1] $, if we let \begin{align*}
c^+&=\inf\{x\in [0,1]\mid x\geq c, x\&x=x\}, \\
c^-&=\sup\{x\in [0,1]\mid x\leq c, x\&x=x\},
\end{align*}
then \begin{itemize}\setlength{\itemsep}{-3pt}
\item Both $ c^+ $ and $c^- $ are idempotent.
\item  If $c$ is idempotent then $ c=c^+=c^- $.
\item  For each non-idempotent element $c$, the restriction of $\&$ on $[c^-,c^+]$ is   either isomorphic to the {\L}ukasiewicz t-norm or to the product t-norm.  Furthermore, the implication operation  in the quantale $([c^-,c^+],\&,c^+)$ is given by \begin{equation*}\label{cal of implication} x\ra^cy=\min\{c^+,x\ra y\}\end{equation*} for all $x,y\in[c^-,c^+]$.  \end{itemize}

The purpose of this section is to characterize the fuzzy lower sets and fuzzy upper sets in $([0,1],d_L)$.  In the case that the t-norm is one of the basic continuous t-norms, such characterizations have been presented in \cite{LZ06}: Proposition \ref{fl1} and Proposition \ref{fu1}. We  use the ordinal sum decomposition of continuous t-norms to establish the desired characterizations in the general case: Proposition \ref{fl} and Proposition \ref{fu}.

\begin{prop} \label{fl1} {\rm(\cite{LZ06}, fuzzy lower sets of $ ([0,1],d_L) $, basic cases)}
Let $\phi: [0,1]\lra[0,1] $ be a map. 	\begin{enumerate}\setlength{\itemsep}{-3pt}
\item[\rm(1)] If $ \& $ is the {\L}ukasiewicz t-norm, then   $\phi$  is a fuzzy lower set of $ ([0,1],d_L) $ if and only if it is decreasing and 1-Lipschitz.
\item[\rm(2)] If $ \& $ is the product t-norm, then   $\phi$  is a fuzzy lower set of $ ([0,1],d_L) $ if and only if it is decreasing and  $y/x\leq \phi(x)/\phi(y)$ whenever $ x>y $.
\item[\rm(3)] If $ \& $ is the t-norm $\min$, then $\phi$  is a fuzzy lower set of $ ([0,1],d_L) $ if and only if it is decreasing and for all $ x\in[0,1]$, $\phi(x)\leq x$ implies that $\phi(x)=\phi(1) $.
	\end{enumerate}
\end{prop}

\begin{prop}\label{fl} {\rm(Fuzzy lower sets of $ ([0,1],d_L) $, general case)}
Let $\&$ be a continuous t-norm on $[0,1]$. Then a map $\phi: [0,1]\lra[0,1]$ is a fuzzy lower set of $ ([0,1],d_L) $ if and only if $ \phi $ satisfies the following conditions:																			\begin{enumerate}\setlength{\itemsep}{-3pt}
		\item[\rm(L1)] If $ a\leq b $ then $ \phi(a)\geq\phi(b) $.
		\item[\rm(L2)] If $ \phi(c)\leq c^- $ then $ \phi(c)=\phi(1) $.
		\item[\rm(L3)] If $c$ is not idempotent and $\phi(c^-)\geq c^-$, then   $\phi(x)\geq c^-$ for all $ x\in[c^-,c^+] $ and  the correspondence $x\mapsto c^+\wedge\phi(x)$ defines a fuzzy lower set of the fuzzy ordered set $([c^-,c^+],d_L^c)$ valued in the quantale $([c^-,c^+],\&,c^+)$, where $d_L^c(x,y)=x\ra^cy=\min\{c^+,x\ra y\}$.  \item[\rm(L4)] If $c$ is idempotent and $\phi(c)\geq c$ then $\phi(1)\geq c$.																																					\end{enumerate}																																						\end{prop}
\begin{proof} 																																							We prove the necessity first.
	
	(L1) Obvious.
	
(L2) This follows from (L1) and that $\phi(c)=\phi(c)\wedge c^-=\phi(c)\&c^-\leq\phi(c)\&(1\ra c)\leq\phi(1).$

(L3) For each $x\in[c^-,c^+]$, since $c^-\leq\phi(c^-)$, then $c^-=c^-\&\phi(c^-)\leq (x\ra c^-)\&\phi(c^-)\leq\phi(x)$.
Hence $\phi(x)\geq c^-$ for all $ x\in[c^-,c^+] $.
It remains to check that																																									$d_L^c(\sigma(a),\sigma(b))\geq d_L^c(b, a)$ whenever $ c^+\geq b>a\geq c^-$.  																																								We calculate: \begin{align*}d_L^c(\sigma(a),\sigma(b))& =\min\{c^+,d_L(\sigma(a),\sigma(b))\}\\ &=\min\{c^+,d_L(c^+\wedge\phi(a),c^+\wedge\phi(b))\} \\ &\geq \min\{c^+,d_L(\phi(a),\phi(b))\}\\ &\geq \min\{c^+,d_L(b,a)\}\\ &= d_L^c(b, a).\end{align*}
	
(L4) Since $c\leq\phi(c)$, then $c=c\&\phi(c)=(1\ra c)\&\phi(c)\leq\phi(1).$
	
Now we prove the sufficiency. That is,   																																					$ \phi(a)\ra\phi(b)\geq b\ra a$ whenever $ b> a $.
	
\textbf{Case 1}. There exists some $c$ such that $ c^-\leq a< b\leq c^+ $.
	
If $ \phi(c^-)\leq c^- $ then $\phi(c^-)=\phi(a)=\phi(b)=\phi(1) $ by  (L2), hence $ \phi(a)\ra\phi(b)=1\geq b\ra a$.

If $ \phi(c^-)> c^- $ then,  by (L3), $\phi(a),\phi(b)\geq c^-$ and \[\min\{c^+, \phi(a)\ra\phi(b)\}=d_L^c(\sigma(a),\sigma(b))\geq d_L^c(b,a)=b\ra a,\] hence $ \phi(a)\ra\phi(b)\geq b\ra a$.
	
\textbf{Case 2}. There exists some idempotent element $c$ such that $ b>c>a $.  It is clear that $b\ra a=a$. Then we proceed with three subcases.
	
\emph{Subcase 1}.  $ \phi(a)\leq a^- $. Then $ \phi(a)=\phi(b)=\phi(1) $ by (L2), hence $\phi(a)\ra\phi(b)=1\geq b\ra a$.
	
\emph{Subcase 2}.   $\phi(b)\geq a $. Then $ \phi(a)\ra\phi(b)\geq \phi(b)\geq a$.
	
\emph{Subcase 3}. $ \phi(a)>a^- $ and $ \phi(b)< a$. First, we show that $a$ is not idempotent and $ \phi(a^+)< a$. If either $a$ is idempotent or $ \phi(a^+)\geq a^+ $, then, by   (L4), we have either $\phi(1)\geq a$ or $ \phi(1)\geq a^+$,   contradicting $ \phi(b)< a $. This shows that $a$ is not  idempotent  and $ \phi(a^+)< a^+$. Then, by (L2), we have $\phi(a^+)=\phi(1)$, hence $\phi(a^+)=\phi(b)<a$. Now, applying (L3) to $a$ we obtain that \begin{align*}a=d_L^a(a^+,a)\leq d_L^a(\phi(a)\wedge a^+,\phi(a^+))&= \phi(a^+)\vee \min\{a^+,\phi(a)\ra \phi(a^+)\},\end{align*} hence   $ \phi(a)\ra\phi(b) =\phi(a)\ra \phi(a^+) \geq a $.
\end{proof}

The following characterization of fuzzy upper sets of $([0,1],d_L) $,   Proposition \ref{fu}, will be used in the next section. Since its proof is similar to that of Proposition \ref{fl}, the details are omitted.

\begin{prop} \label{fu1} {\rm(\cite{LZ06}, fuzzy upper sets of $([0,1],d_L)$, basic cases)} Let $\psi: [0,1]\lra[0,1]$ be a map.
\begin{enumerate}\setlength{\itemsep}{-3pt}
\item[\rm(1)] If $ \& $ is the {\L}ukasiewicz t-norm, then $ \psi $ is a fuzzy upper set of $ ([0,1],d_L) $ if and only if it is increasing and 1-Lipschitz.
\item[\rm(2)] If $ \& $ is the product t-norm, then $ \psi $ is a fuzzy upper set of $ ([0,1],d_L) $ if and only if it is increasing and $y/x\leq \psi(y)/\psi(x)$  whenever $ x>y $.
\item[\rm(3)] If $ \& $ is the t-norm $\min$, then $ \psi $ is a fuzzy upper set of $ ([0,1],d_L) $ if and only if it is increasing and  for every $ x\in[0,1]$, $\psi(x)< x$ implies $\psi(x)=\psi(1)$.
\end{enumerate}
\end{prop}

\begin{prop}\label{fu} {\rm(Fuzzy upper sets in $ ([0,1],d_L) $, general case)}
A map $ \psi:[0,1]\lra[0,1] $ is a fuzzy upper set of $ ([0,1],d_L) $ if and only if $ \psi $ satisfies the following conditions:
\begin{enumerate}\setlength{\itemsep}{-3pt}
\item[\rm(U1)] If $ a\leq b $ then $ \psi(a)\leq\psi(b) $.
\item[\rm(U2)] If $ \psi(c)< c^- $ then $ \psi(c)=\psi(1) $.
\item[\rm(U3)] If $ c $ is not idempotent and $ \psi(c^-)\geq c^- $, then $ \psi(x)\geq c^- $ for all $ x\in[c^-,c^+] $ and the correspondence $x\mapsto c^+\wedge\psi(x)$ defines a fuzzy upper set of the fuzzy ordered set $([c^-,c^+],d_L^c)$ valued in the quantale $([c^-,c^+],\&,c^+)$, where $d_L^c(x,y)=x\ra^cy=\min\{c^+,x\ra y\}$.
\end{enumerate}
\end{prop}

\section{Flat ideals in $([0,1],d_L)$}

In this section, we characterize flat ideals in $([0,1],d_L)$. The strategy is to do this in the case that $\&$ is one of the   basic continuous t-norms, then in the general case with help of the ordinal sum decomposition of continuous t-norms.

Given a fuzzy ordered set $ X $, a net $ \{x_\lambda\}_{\lambda\in D}$ in $ X $ is forward Cauchy \cite{Wagner97} if \[ \bigvee_{\lambda\in D}\bigwedge_{\sigma\geq\mu\geq\lambda}X(x_{\mu},x_{\sigma})=1. \]
A fuzzy lower set $ \phi $ of a fuzzy ordered set $X$ is a forward Cauchy ideal if there exists some forward Cauchy net $ \{x_{\lambda}\}_{\lambda\in D} $ such that \[ \phi=\bigvee_{\lambda\in D}\bigwedge_{\mu\geq\lambda}X(-,x_{\mu}). \]

The following theorem was proved in \cite{SV2005} in the case that $\&$ is isomorphic to the product t-norm, and in \cite{LLZ17} in the case that $\&$ is isomorphic to the {\L}ukasiewicz t-norm.

\begin{thm}{\rm(\cite{SV2005,LLZ17})} \label{flat=forward Cauchy}
Let $ \& $ be a continuous   t-norm on $[0,1]$ that is  either isomorphic to the product t-norm or to the {\L}ukasiewicz t-norm, and  $\CQ$ be the quantale $([0,1],\&,1)$. Then for each $\CQ$-ordered set $X$,  flat ideals  coincide with forward Cauchy ideals.
\end{thm}
\begin{thm}\label{fib1} {\rm(Flat ideals in $ ([0,1],d_L) $, basic cases)}
\begin{enumerate}\setlength{\itemsep}{-3pt}
\item[\rm(1)] If  $\&$ is a continuous   t-norm on $[0,1]$ that is   either isomorphic to the {\L}ukasiewicz t-norm or to the product t-norm, then for every flat ideal $ \phi $ of $ ([0,1],d_L) $, there exists some $x\in[0,1] $ such that $ \phi=d_L(-,x) $.
\item[\rm(2)] If $ \& $ is the t-norm $\min$, then a fuzzy lower set $ \phi$ of $ ([0,1],d_L) $ is a flat ideal if and only if $\phi(0)=1$.
\end{enumerate}																											\end{thm}
																												\begin{proof}																													(1) By Theorem \ref{flat=forward Cauchy}, \[ \phi(x)=\bv_{\lambda\in D}\bw_{\mu\geq \lambda}(x\ra x_{\mu}) \] for some forward Cauchy net $ \{x_i\}_{i\in I} $.
Then $ x=\bv_{\lambda\in D}\bw_{\mu\geq\lambda}x_{\mu} $ satisfies the requirement.

(2) Trivial.																													\end{proof}
																													
																													%\begin{rem}\label{fib1r} 																													Similar as Remark \ref{fos12} we can define flat idea on $ ([\alpha_i,\beta_i],d_L') $, and according to Theorem \ref{osd}, the restriction of $\&$ on $[\alpha_i,\beta_i]$ is either isomorphic to the \L ukasiewicz t-norm or to the product t-norm. So, given a flat ideal $ \sigma $ of $ ([\alpha_i,\beta_i],d_L') $ then there exists some $ d\in[\alpha_i,\beta_i] $ such that $ \sigma(x)=d_L'(x,d)=(x\ra d)\wedge c^+$. 																														\end{rem}

\begin{thm}\label{fidealinL} {\rm(Flat ideals in $ ([0,1],d_L) $, general case)}  Let $\&$ be a continuous t-norm on $[0,1]$. Then  a fuzzy lower set $ \phi $ of $ ([0,1],d_L) $ is a flat ideal if and only if it satisfies:													\begin{enumerate}\setlength{\itemsep}{-3pt}
\item[\rm(F1)]   $\phi(0)=1$.
\item[\rm(F2)] For each $c\in[0,1]$,  if $ \phi(c)\geq c^+ $  then $ \phi(c) $ is idempotent.
\item[\rm(F3)]   For each $c\in[0,1]$,  if $c$ is not idempotent and $\phi(c^-)> c^-$, then the correspondence $x\mapsto c^+\wedge\phi(x)$ defines  a flat ideal in the fuzzy ordered set $([c^-,c^+],d_L^c)$ valued in the quantale $([c^-,c^+],\&,c^+)$, where $d_L^c(x,y)=x\ra^cy=\min\{c^+,x\ra y\}$.
\end{enumerate}
\end{thm}

\begin{proof} The proof is divided into four lemmas given below.
The necessity part  is contained in  Lemma \ref{f1} and Lemma \ref{f2};   the sufficiency part follows  from Lemma \ref{compute1} and Lemma \ref{compute2}.	\end{proof}

\begin{lem}\label{f1} 	Let $ \phi $ be a flat ideal in $ ([0,1],d_L) $. Then for each $ c\in[0,1] $, if $ \phi(c)\geq c^+ $ then $ \phi(c) $ is idempotent.
\end{lem}
\begin{proof} Consider the fuzzy upper sets $ \psi_1,\psi_2 $  of $ ([0,1],d_L) $ given by $ \psi_1\equiv\phi(c) $ and $ \psi_2=c\ra\id $. 	Since $ \phi(0)=1 $, then $ \phi\otimes\psi_1=\phi(c) $. Since $ \phi $ is fuzzy lower set of $ ([0,1],d_L) $, it follows that \[\phi\otimes\psi_2=\bv_{x\in[0,1]}\phi(x)\&(c\ra x)=\phi(c).\] Hence $(\phi\otimes\psi_1)\wedge(\phi\otimes\psi_2)=\phi(c)$.
	
Since $\phi(c)\geq c^+ $, then \[ (\psi_1\wedge\psi_2)(x)=\begin{cases}
	c\ra x,&0\leq x<c,\\
	\phi(c),&c\leq x\leq 1,
	\end{cases} \]
	hence \begin{align*}
\phi(c)&=(\phi\otimes\psi_1)\wedge(\phi\otimes\psi_2)\\ &	= \phi\otimes(\psi_1\wedge\psi_2)\\ &= \Big(\bv_{x\in[0,c)}\phi(x)\&(c\ra x)\Big)\vee\Big(\bv_{x\in[c,1]}\phi(x)\&\phi(c)\Big)\\
	%&\leq\big[\bv_{x\in[0,c)}(c\ra x)\big]~\vee~[\phi(c)\&\phi(c)]&&(\phi\text{ is decreasing})\\
	&\leq c^+\vee(\phi(c)\&\phi(c))\\
	&=\phi(c)\&\phi(c),&&(\phi(c)\geq c^+)
	\end{align*}
showing that $\phi(c)$ is idempotent.
\end{proof}

\begin{lem}\label{f2}	Let $ \phi $ be a flat ideal in $ ([0,1],d_L) $. For each $c\in[0,1] $, if $c$ is not idempotent and $\phi(c^-)> c^-$  then   $\sigma(x)= c^+\wedge\phi(x)$ is  a flat ideal in the fuzzy ordered set $([c^-,c^+],d_L^c)$ valued in   $([c^-,c^+],\&,c^+)$, where $d_L^c(x,y)=x\ra^cy=\min\{c^+,x\ra y\}$.
\end{lem}
\begin{proof}
Since $\phi(c^-)>c^-$, then $\phi(c^-)$ is idempotent by Lemma \ref{f1}. Thus, $\phi(c^-)\geq c^+$ and then $\sigma(c^-)=c^+$, showing that $\sigma$ is inhabited.
It remains to show that for any  fuzzy upper sets $\psi_1,\psi_2$ of the fuzzy ordered set $([c^-,c^+],d_L^c)$ (valued in  $([c^-,c^+],\&,c^+)$), \[\sigma\otimes(\psi_1\wedge\psi_2)=(\sigma\otimes\psi_1)\wedge(\sigma\otimes\psi_2).\]
	
\textbf{Step 1}.	For each fuzzy upper set $\psi$ in $([c^-,c^+],d_L^c)$, the map $ \overline{\psi}:[0,1]\lra[0,1] $, given by
	\[ \overline{\psi}(x)=\begin{cases}
	x,&x\in[0,c^-),\\
	\psi(x),&x\in[c^-,c^+],\\
	\psi(c^+),&x\in(c^+,1],
	\end{cases} \] is a fuzzy upper set of $ ([0,1],d_L)$.

This follows from Proposition \ref{fu} immediately.

\textbf{Step 2}. For each fuzzy upper set $\psi$ in $([c^-,c^+],d_L^c)$, $\sigma\otimes\psi=\phi\otimes\overline{\psi}.$	

We calculate:  \begin{align*} \phi\otimes\overline{\psi}& = \bv_{x\in [0,1]}\phi(x)\&\overline{\psi}(x) \\ &=\Big(\bv_{x\in[0,c^-)}\phi(x)\&x\Big)\vee \Big(\bigvee_{x\in[c^-,c^+]}\phi(x)\& \psi (x)\Big)\vee \Big(\bv_{x\in(c^+,1]}\phi(x)\& \psi (c^+)\Big)\\ & \leq c^-\vee\Big(\bigvee_{x\in[c^-,c^+]}(\phi(x)\wedge c^+)\&\psi(x)\Big) \vee\big(\phi(c^+)\&\psi(c^+)\big)\\ & ~~~~~~~~~ \big(\text{$\psi(x)\in[c^-,c^+]$ for all $x\in[c^-,c^+]$}\big)\\ & =\bigvee_{x\in[c^-,c^+]}\sigma(x)\&\psi(x)\\ & ~~~~~~~~~ \big(\text{$c^-\leq\phi(c^-)\& \psi(c^-)$, $(\phi(c^+)\wedge c^+)\&\psi(c^+)=\phi(c^+)\&\psi(c^+)$}\big) \\ &=\sigma\otimes \psi.  \end{align*}
 	
\textbf{Step 3}.	For any   fuzzy upper sets $\psi_1,\psi_2$ of $([c^-,c^+],d_L^c)$,
	\begin{align*}
	\sigma\otimes(\psi_1\wedge\psi_2)&=\phi\otimes(\overline{\psi_1\wedge\psi_2})\\ &=\phi\otimes(\overline{\psi_1}\wedge\overline{\psi_2})\\
	&=(\phi\otimes\overline{\psi_1})\wedge(\phi\otimes\overline{\psi_2})&(\phi\text{ is flat})\\
	&=(\sigma\otimes \psi_1)\wedge(\sigma\otimes\psi_2).
	\end{align*}

Therefore,	$ \sigma $ is a flat ideal in the fuzzy ordered set $([c^-,c^+],d_L^c)$ valued in  the quantale $([c^-,c^+],\&,c^+)$.
\end{proof}

The verification of the following lemma is straightforward and is left to the reader.

\begin{lem}\label{compute1}
Let $ K$ be a subset of $[0,1] $, $ \phi:K\lra[0,1] $ be a decreasing map, and $ \psi_1,\psi_2: K \lra [0,1] $ be increasing maps. Then \[ \bv_{x\in K}\phi(x)\wedge\psi_1(x)\wedge\psi_2(x)=\Big(\bv_{x\in K}\phi(x)\wedge\psi_1(x)\Big)\wedge\Big(\bv_{x\in K}\phi(x)\wedge\psi_2(x)\Big).\]
\end{lem}

%\begin{proof} The inequality that 	\[ \bv_{x\in K}\phi(x)\wedge\psi_1(x)\wedge\psi_2(x)\leq\big[\bv_{x\in K}\phi(x)\wedge\psi_1(x)\big]\wedge\big[\bv_{x\in K}\phi(x)\wedge\psi_2(x)\big] \] is trivial.
	
%To see the converse, 	\begin{align*} 	&\big[\bv_{x\in K}\phi(x)\wedge\psi_1(x)\big]\wedge\big[\bv_{x\in K}\phi(x)\wedge\psi_2(x)\big]\\ 	=&\bv_{x,y\in K}\phi(x)\wedge\psi_1(x)\wedge\phi(y)\wedge\psi_2(y)\\ 	=&\big[\bv_{x\leq y}\phi(x)\wedge\phi(y)\wedge\psi_1(x)\wedge\psi_2(y)\big]~\vee~\big[\bv_{y\leq x}\phi(x)\wedge\phi(y)\wedge\psi_1(x)\wedge\psi_2(y)\big]\\ 	\leq&\bv_{x\in K}\phi(x)\wedge\psi_1(x)\wedge\psi_2(x). 	\end{align*} \end{proof}

\begin{lem}\label{compute2}
Suppose $\phi$ is a fuzzy lower set of $([0,1],d_L)$ that satisfies {\rm(F2)} and  {\rm(F3)} in Theorem \ref{fidealinL}. Let $K_{\phi}=\{x\in[0,1]\mid\phi(x)\geq x^+\}$. Then for each fuzzy upper set $\psi$ of $([0,1],d_L)$, \[\phi\otimes\psi=\bv_{x\in{K_{\phi}}}\phi(x)\wedge\psi(x).\]
\end{lem}
\begin{proof}
For all $x\in K_{\phi}$,   $\phi(x)$ is idempotent by (F2). Thus, $\phi(x)\&\psi(x)=\psi(x)\wedge\psi(x)$ whenever $x\in K_\phi$. Since \[\phi\otimes\psi=\bv_{x\in[0,1]}\phi(x)\&\psi(x),\] it suffices to show that for all $c\not\in K_{\phi}$, there is some $b\in K_{\phi}$ such that $\phi(c)\&\psi(c)\leq \phi(b)\wedge\psi(b)$. We proceed with three cases.
	
\textbf{Case 1}. $c$ is non-idempotent  and $\phi(c^-)>c^-$.
	
By (F3), $\sigma(x)=c^+\wedge\phi(x)$ is a flat ideal in the fuzzy ordered set $([c^-,c^+],d_L^c)$. 	Since the restriction of $\&$ on $[c^-,c^+]$ is either isomorphic to the \L ukasiewicz t-norm or to the product t-norm, by Theorem \ref{fib1} there is some $b\in[c^-,c^+] $ such that $\sigma(x)=d_L^c(x,b)$. We claim that this $b$ satisfies the requirement, that is, $\phi(c)\&\psi(c)\leq \phi(b)\wedge\psi(b)$.
	
If $c\leq b$, then	$\sigma(c)=d_L^c(c,b)=c^+$, hence $\phi(c)\geq c^+$, showing that   $c\in K_{\phi}$, contradictory to that $c\notin K_\phi$. Hence we have $b<c<c^+$ and $\phi(c)<c^+$. Consequently, $\phi(c)=\sigma(c) = c\ra b$. Then \[\phi(c)\&\psi(c)\leq(c\ra b)\leq c^+=\sigma(b)\leq\phi(b)\] and,  since $\psi$ is a fuzzy upper set, \[\phi(c)\&\psi(c)=(c\ra b)\&\psi(c)\leq\psi(b).\] Hence $\phi(c)\&\psi(c)\leq \phi(b)\wedge\psi(b)$, as desired.

\textbf{Case 2}. $c$ is non-idempotent   and $\phi(c^-)\leq c^-$.
We distinguish two subcases.

\emph{Subcase 1}. $\phi(c)$ is idempotent.
We show that $ b=\phi(c) $ satisfies the requirement.
Since $ \phi(c^-)\leq c^- $, then $b=\phi(c)\leq \phi(c^-)\leq c^-<c $, hence $ \phi(b)\geq\phi(c)=b =b^+$,  showing that $ b\in{K_{\phi}} $.
If $ \psi(b)\geq b$, then $\phi(b)\&\psi(b)\geq b\&b=b\geq\phi(c)\&\psi(c)$.
If $\psi(b)<b$, then $\psi(b)=\psi(c)$ by (U2) in Proposition \ref{fu}, hence $\phi(b)\&\psi(b)\geq\phi(c)\&\psi(c)$. So, $ b$ satisfies the requirement.

\emph{Subcase 2}. $\phi(c)$ is non-idempotent.

Since $ \phi(c^-)\leq c^- $, by (L2) in Proposition \ref{fl}  we have  $\phi(c^-) =\phi(1) $, hence $\phi(c)=\phi(1)$. Now we use this fact to show that $ \phi(\phi(c)^-)>\phi(c)^- $.
Since $ \phi(\phi(c)^+)\ra\phi(1)\geq 1\ra\phi(c)^+ =\phi(c)^+$, it follows that \[\phi(c)^+\&\phi(\phi(c)^+)=\phi(c)^+\wedge\phi(\phi(c)^+)\leq\phi(1),\] hence $\phi(\phi(c)^+)\leq\phi(1)$ because $\phi(c)^+>\phi(c)=\phi(1)$, so,   \[\phi(\phi(c)^+)=\phi(1) =\phi(c).\] Thus,  \[ \phi(\phi(c)^-)\geq\phi(\phi(c)^+)=\phi(c)>\phi(c)^-.\]

By (F3), $\sigma(x)=\phi(c)^+\wedge\phi(x)$ is a flat ideal in the fuzzy ordered set $([\phi(c)^-,\phi(c)^+],d_L^{\phi(c)})$ (valued in  $([\phi(c)^-,\phi(c)^+],\&,\phi(c)^+)$.
Then there is some $b\in[\phi(c)^-,\phi(c)^+] $ such that $\sigma(x)=d_L^{\phi(c)}(x,b)$. Since $ \phi(\phi(c)^+)=\phi(c) $, it follows that $\sigma(\phi(c)^+)=\phi(c)$, hence $ b<\phi(c)^+ $.
We claim that this $b$ satisfies the requirement. It suffices to show that $ b\in K_{\phi} $ and \[\phi(c)\&\psi(c)\leq\phi(\phi(c)^+)\&\psi(\phi(c)^+)\leq\phi(b)\wedge\psi(b).\]

Firstly, since $ \sigma(b)=\phi(c)^+ $, then $ \phi(b)\geq\phi(c)^+=b^+$, hence $b\in K_{\phi} $.

Secondly, we check the inequality $\phi(c)\&\psi(c)\leq\phi(\phi(c)^+)\&\psi(\phi(c)^+)$.  If $ \psi(\phi(c)^+)\geq \phi(c)^+$, then \[\phi(\phi(c)^+)\&\psi(\phi(c)^+)=\phi(c)\geq\phi(c)\&\psi(c).\]
If $\psi(\phi(c)^+)<\phi(c)^+$, then by (U2) in Proposition \ref{fu}  we have  $\psi(\phi(c)^+)=\psi(c)$, hence \[\phi(\phi(c)^+)\&\psi(\phi(c)^+)\geq\phi(c)\&\psi(c).\]

Thirdly, we check the inequality $\phi(\phi(c)^+)\&\psi(\phi(c)^+)\leq\phi(b)\wedge\psi(b)$. Since $\phi(\phi(c)^+)=\phi(c)$, it follows that $\phi(\phi(c)^+)=\sigma(\phi(c)^+)  =\phi(c)^+\ra b $, then   \[ \phi(\phi(c)^+)\&\psi(\phi(c)^+)=(\phi(c)^+\ra b)\&\psi(\phi(c)^+)\leq\psi(b). \]
Since
\[\phi(\phi(c)^+)\&\psi(\phi(c)^+)=\phi(c)\&\psi(\phi(c)^+)\leq\phi(c)\leq \phi(c)^+=\sigma(b)\leq \phi(b),   \] then the desired inequality  follows.

\textbf{Case 3}. $c$ is idempotent. We also distinguish two subcases.

\emph{Subcase 1}. $\phi(c)$ is idempotent.
We claim that $b=\phi(c)$ satisfies the requirement.
Since $c\notin K_\phi$, then $b=\phi(c)<c$, hence $ \phi(b)\geq\phi(c)=b $. This shows  $b\in K_{\phi}$.
If $ \psi(b)\geq b$, then $\phi(b)\&\psi(b)\geq b\&b=b\geq\phi(c)\&\psi(c)$.
If $\psi(b)<b$, then by (U2) in Proposition \ref{fu} we have that $\psi(b)=\psi(c)$, hence $\phi(b)\&\psi(b)\geq\phi(c)\&\psi(c)$.

\emph{Subcase 2}. $\phi(c)$ is non-idempotent.
Since $c\not\in K_{\phi}$, then $ \phi(c)<c $, hence $ \phi(c)=\phi(1) $ by (L2) in Proposition \ref{fl}.
The  rest of the proof then proceeds in the same way as that for \emph{Subcase 2} in \textbf{Case 2},  the details are omitted here. \end{proof}

\end{document}